\documentclass[12pt]{amsart}

\usepackage{amsfonts,amsmath,latexsym,amssymb,verbatim,amsbsy,amsthm}

\usepackage{soul}

\usepackage[top=1in, bottom=1in, left=1in, right=1in]{geometry}

\usepackage[dvipsnames]{xcolor}

\usepackage[colorlinks=true, pdfstartview=FitV, linkcolor=RoyalBlue,citecolor=ForestGreen, urlcolor=blue]{hyperref}

\numberwithin{equation}{section}

\renewcommand{\geq}{\geqslant}
\renewcommand{\ge}{\geqslant}
\renewcommand{\leq}{\leqslant}
\renewcommand{\le}{\leqslant}

\theoremstyle{plain}
\newtheorem{THEOREM}{Theorem}[section]

\newtheorem{theorem}[THEOREM]{Theorem}
\newtheorem{corollary}[THEOREM]{Corollary}
\newtheorem{lemma}[THEOREM]{Lemma}

\theoremstyle{definition}

\theoremstyle{remark}

\newtheorem{remark}[THEOREM]{Remark}

\newtheorem{example}[THEOREM]{Example}

\newcommand{\thm}[1]{Theorem~\ref{#1}}
\newcommand{\lem}[1]{Lemma~\ref{#1}}

\newcommand{\cor}[1]{Corollary~\ref{#1}}

\def \a {\alpha} 

\def \d {\delta}

\def \g {\gamma}

\def \e {\varepsilon}

\def \L {\Lambda}

\def \r {\rho}
\def \th {\theta}

\def \cL {\mathcal{L}}

\def \cM {\mathcal{M}}

\def \cP {\mathcal{P}}

\def \cF {\mathcal{F}}

\def \aL {{\mathcal L}}
\def \aT {{\mathcal T}}

\newcommand{\Z}{\ensuremath{\mathbb{Z}}}   
\newcommand{\R}{\ensuremath{\mathbb{R}}}   
\newcommand{\T}{\ensuremath{\mathbb{T}}}   

\def \ra {\rightarrow}

\def \rmax  {\rho_+}
\def \rmin {\rho_-}
\def \umax  {u_+}
\def \umin {u_-}
\def \xmax {x_+}
\def \xmin {x_-}

\def \dx  {\, \mbox{d}x}
\def \dy  {\, \mbox{d}y}
\def \dz  {\, \mbox{d}z}
\def \ds  {\, \mbox{d}s}
\def \dw  {\, \mbox{d}w}
\def \dth  {\, \mbox{d}\th}
\def \ddt  {\frac{\mbox{d\,\,}}{\mbox{d}t}}
\def \DDt  {\frac{\mbox{D\,\,}}{\mbox{D}t}}
\def \dD  {\mbox{D}}

\def \aL {\mathcal{L}}
\newcommand{\supr}{\text{supp}\ensuremath{\,\rho(\cdot,t)}}   

\begin{document}

\title[Eulerian dynamics with a commutator forcing]{Eulerian dynamics with \\ a commutator forcing II: flocking}

\author{Roman Shvydkoy}
\address{Department of Mathematics, Statistics, and Computer Science, M/C 249,\\
University of Illinois, Chicago, IL 60607, USA}
\email{shvydkoy@uic.edu}

\author{Eitan Tadmor}
\address{Department of Mathematics, Center for Scientific Computation and Mathematical Modeling (CSCAMM), and Institute for Physical Sciences \& Technology (IPST), University of Maryland, College Park\newline
Current address: Institute for Theoretical Studies (ITS), 
ETH-Zurich, Clausiusstrasse 47, CH-8092 Zurich, Switzerland}
\email{tadmor@cscamm.umd.edu}

\date{\today}



\thanks{\textbf{Acknowledgment.} Research was supported in part by NSF grants DMS16-13911, RNMS11-07444 (KI-Net) and ONR grant N00014-1512094 (ET) and by NSF grant DMS 1515705 (RS). RS thanks CSCAMM for the hospitality for his March 2016 visit which initiated this project. Both authors thank the Institute for Theoretical Studies (ITS) at ETH-Zurich for the hospitality.}

\begin{abstract}
We continue our study of one-dimensional class of Euler equations, introduced in \cite{ST2016}, driven by a forcing with a commutator structure of the form $[\aL_\phi,u](\rho)=\phi*(\rho u)- (\phi*\rho)u$, where $u$ is the velocity field and $\phi$ belongs to a rather general class of \emph{influence} or interaction kernels.

In this paper we quantify the large-time behavior of such systems in terms of \emph{fast flocking} for two prototypical sub-classes of kernels: bounded positive $\phi$'s, and singular $\phi(r) = r^{-(1+\a)}$ of order $\alpha\in [1,2)$ associated with the action of the fractional Laplacian $\aL_\phi=-(-\partial_{xx})^{\alpha/2}$. Specifically, we prove  fast velocity alignment as the velocity $u(\cdot,t)$ approaches  a constant state, $u \to \bar{u}$, with exponentially decaying slope and curvature bounds $|u_x(\cdot,t)|_{\infty}+ |u_{xx}(\cdot,t)|_{\infty}\lesssim e^{-\d t}$. The alignment is accompanied by exponentially fast flocking of the density towards a fixed traveling state $\rho(\cdot,t) - {\rho_{\infty}}(x - \bar{u} t) \rightarrow 0$.
\end{abstract}

\maketitle
\setcounter{tocdepth}{1}
\tableofcontents

\section{Introduction and statement of main results.}

\subsection{Flocking hydrodynamics}
In this paper we continue our study initiated in \cite{ST2016}, of Eulerian dynamics driven by forcing with a commutator structure. In the one-dimensional case, the dynamics of a velocity $u:\Omega\times \R_+ \mapsto \R$ is governed by the system of form   
\begin{equation}\label{eq:FH}
\left\{
\begin{split}
\rho_t +(\rho u)_x&=0,\\
u_t+uu_x &=[\aL_\phi,u](\rho),
\end{split}\right. 
\end{equation}
where  the commutator on the right,
$[\aL_\phi,u](\rho):=\aL_\phi(\rho u)-\aL_\phi(\rho)u$, involves a convolution kernel
\begin{equation}\label{eq:aL}
\aL_\phi(f):=\int_{\R} \phi(|x-y|) (f(y)-f(x))dy.
\end{equation}
The motivation for \eqref{eq:FH} comes from the hydrodynamic description of a large-crowd dynamics driven by Cucker-Smale agent-based model
\begin{equation}\label{eq:CS}
\left\{
\begin{split}
\dot{x}_i&=v_i,\\
\dot{v}_i& =\frac{1}{N}\sum_{j=1}^N \phi(|x_i-x_j|)(v_j-v_i), 
 \end{split}\right.
\qquad (x_i,v_i)\in \Omega\times \R, \quad i=1,2, \ldots, N.
\end{equation}
Here, $\phi$ is a positive, bounded influence function which models the binary interactions among agents in $\Omega$. We focus our attention on the periodic or open line setup, $\Omega=\T, \R$.  For large crowds, $N\gg 1$, the dynamics can be encoded in terms of the empirical distribution $f_N= \frac{1}{N}\sum_{i=1}^N \delta_{x_i}(x)\otimes \delta_{v_i}(v)$, so that its limiting moments lead to a density, $\displaystyle \rho(x,t)=\lim_{N\rightarrow \infty} \int_{\R} f_N(x,v,t)dv$, and momentum, $\displaystyle  \rho u(x,t)=\lim_{N\rightarrow \infty} \int_{\R} vf_N(x,v,t)dv$, governed by \eqref{eq:FH}, \cite{HT2008,CCP2017}.

\begin{equation}\label{e:phi}
\left\{
\begin{split}
\rho_t + (\rho u)_x&=0,\\
u_t+u u_x &=\int_{\R}\phi(|x-y|)(u(y,t)-u(x,t))\rho(y,t) \dy
\end{split}\right. \qquad (x,t): \Omega \times [0, \infty).
\end{equation}

The other important limit of such systems --- their large time behavior for $t\gg 1$, is described by the \emph{flocking} phenomenon.  To this end, let us introduce the set of flocking state solutions, consisting of constant velocities, $\bar{u}$, and  traveling density waves 
$\bar{\rho}=\rho_\infty(x-t\bar{u})$,
\begin{equation}\label{ }
\cF = \{ ( \bar{u}, \bar{\rho}):  \bar{u} \equiv \mbox{constant}, \bar{\rho}(x,t) = \rho_\infty(x-t \bar{u}) \}. 
\end{equation}
We say that a solution $(u(\cdot,t),\rho(\cdot,t))$ converges to a flocking state  $(\bar{u},\bar{\rho}) \in \cF$ in space $X\times Y$ if 
\[
\| u(\cdot,t)- \bar{u}\|_X + \|\rho(\cdot,t) - \bar{\rho}(\cdot,t)\|_Y  \to 0, \text{ as }  t\to \infty.
\]
 This represents the process of \emph{alignment} where the diameter of velocities  tends to zero 
 \begin{subequations}\label{eqs:flocking}
 \begin{equation}\label{eq:Vtoz}
 V(t):=\max_{x,y \in \supr}|u(x,t)-u(y,t)| \to  0, \text{ as }  t\to \infty.
 \end{equation}
 In particular, there is a \emph{fast alignment} if the flocking convergence rate is exponential. In the present case of symmetric interactions, the conservation of averaged mass and momentum, 
\[
\cM(t):=\frac{1}{2\pi}\int_{\T} \rho(x,t)dx \equiv \cM_0, \qquad 
 \cP(t):=\frac{1}{2\pi} \int_{\T} (\rho u)(x,t)dx  \equiv \cP_0
 \]
  implies that a limiting flocking velocity, provided it exists, is given by $\bar{u}=\cP_0/\cM_0$.
\begin{remark} In the case when the dynamics of \eqref{e:phi} takes place over the line  $\Omega=\R$ as in \cite{HT2008,TT2014},  the flocking phenomenon assumes a compactly supported initial configuration with finite initial velocity variation,  
\[
  D_0 :=\max_{x,y \in \text{supp}(\rho_0)}|x-y| <\infty, \qquad 
 V_0 :=\max_{x,y \in \text{supp}(\rho_0)}|u_0(x)-u_0(y)| <\infty.
\]
It requires that, in addition to \eqref{eq:Vtoz}, the flow remains compactly supported 
\begin{equation}\label{e:flockR}
D(t)\leq D_\infty <\infty, \qquad D(t):=\max_{x,y \in \supr}|x-y|
\end{equation}
 This reflects the corresponding flocking behavior in the   agent-based Cucker-Smale model, $\displaystyle \max_{1\leq i,j\leq N} |x_i(t)-x_j(t)|\leq D_\infty$ and $\displaystyle \max_{1\leq i,j\leq N} |v_i(t)-v_j(t)| \to 0$ as $t \to \infty$, \cite{HT2008,MT2014}.
\end{remark}  
\end{subequations}

\subsection{Smooth solutions must flock}
The flocking hydrodynamics of \eqref{e:phi} for bounded positive $\phi$'s follows, as long as they admit global smooth solutions. Indeed, the statement that ``smooth solutions must flock'' holds in the general setup of positive kernels whether symmetric or not  \cite[Lemma 3.1]{ISV2016}, \cite[Theorem 2.1]{TT2014}. For the sake of completeness we include below the proof of flocking along the lines of \cite[theorem 2.3]{MT2014}  which is stated in the following lemma for the periodic case $\Omega = \T$.

\begin{lemma}[Smooth solutions must flock]\label{l:V}
Let $(\rho,u)$ be a smooth solution of the one-dimensional system
\begin{equation}\label{}
\left\{
\begin{split}
\rho_t + (\rho u)_x&=0,\\
u_t+u u_x &=\int_{\R}k(x,y,t)(u(y,t)-u(x,t))\rho(y,t) \dy
\end{split}\right. \qquad (x,t): \T \times [0, \infty),
\end{equation}
with strictly positive kernel, $\displaystyle  \iota_k(t) = \inf_{x,y\in \T} k(x,y,t) >0$.
Then there is a flocking alignment
\[
V(t) \leq V(0) exp\left\{- \cM \int_{\tau=0}^t \!\!\!\iota_k(\tau)d\tau\right\}, \qquad V(t)=\max_{x,y\in\T}|u(x,t)-u(y,t)|.
\]
In  particular, the case of symmetric interaction \eqref{e:phi} admits fast alignment,
\begin{equation}\label{eq:fa}
V(t) \leq V(0) e^{- \cM \iota_\phi t}, \qquad \iota_\phi:=\min_{x\in \T}\phi(|x|).
\end{equation}
\end{lemma}
  
\begin{proof}
Let $\xmin(t)$ be a point where $\displaystyle \umin = u(\xmin(t),t) = \min u$, and  $\xmax(t)$ be a point where $\displaystyle \umax = u(\xmax(t),t) = \max u$.   
Then the maximal value does not exceed,
\[
\ddt \umax  = \int_\T k(\xmax,y,t)(u (y) - \umax) \rho(y,t) \dy  \le \iota_k \int_\T (u (y) -\umax) \rho(y,t) \dy.
\]
Similarly, we have the lower bound
\[ 
\ddt \umin \ge  \iota_k  \int_\T (u (y) -\umin) \rho(y,t) \dy.
 \] 
Subtracting the latter implies that the velocity diameter $\displaystyle V(t)=\max_{x,y\in\T}|u(x,t)-u(y,t)|$ satisfies
\[
\ddt V(t) \le -   \iota_k \cM V(t), \qquad V(t)=\umax(t)-\umin(t)
\]
 and the result readily follows.
\end{proof}

We demonstrate the generality of lemma \ref{l:V} with the following two examples.   
\begin{example}[an example on non-symmetric kernel] The Mostch-Tadmor model \cite{MT2011} uses an \emph{adaptive} normalization, where the pre-factor $1/N$ on the right of \eqref{eq:CS} is replaced  
by $1/\sum_j \phi(|x_i-x_j|)$, leading to the flocking hydrodynamics with \emph{non-symmetric} kernel $k(x,y,t)=\phi(|x-y|)/(\phi*\rho)(x,t)$. The lower-bound
\[
k(x,y,t)=\frac{\phi(|x-y|)}{(\phi*\rho)(x,t)} \geq \frac{\iota_\phi}{I_\phi \cM}, \qquad I_\phi=\max_{x\in \T} \phi(|x|).
\]
shows that flocking holds for positive, bounded $\phi$'s, with exponential rate dictated by the condition number of $\phi$ but otherwise independent of the total mass,
$\displaystyle V(t) \leq V(0)e^{-(\iota_\phi/I_\phi) t}$.
\end{example}

\begin{example}[an example of unbounded kernels]\label{rem:uk} The fractional Laplacian $\aL_\a:=\aL_{\phi_\a}$ is associated with the  singular periodized kernels 
\begin{equation}\label{e:phia}
\phi_\a(x) = \sum_{k \in \Z}  \frac{1}{|x+ 2\pi k|^{1+\a}}, \text{ for } 0<\a<2.
\end{equation}
 Since the argument of lemma \ref{l:V} does not use local integrability, it applies in the present setting with $\iota_\a=\inf_x \phi_\a(|x|)>0$, leading to fast alignment \eqref{eq:fa}.
\end{example}

We close this subsection by noting that the the extension of lemma \ref{l:V} to the case of open space $\Omega = \R$ was proved in \cite{TT2014}. To this end one restricts attention to the  dynamics over $\{\supr\}$: the growth of the velocity diameter $\displaystyle V(t):=\mathop{\max}_{x,y\in\supr}\!\!|u(x,t)-u(y,t)|$,
\[
\ddt V(t) \le -   \iota_k \cM V(t), \qquad V(t)=\max_{x\in\supr}\!\!\!\!\!u(x,t)-\min_{y\in \supr}\!\!\!\!\!u(y,t),
\]
is coupled with the obvious bound on the growth of the  density support, $\ddt D(t) \leq V(t) D(t)$. Assume that $\phi$ is decreasing so that $\iota_\phi \geq \phi(D(t))$. It implies a decreasing free energy ${\mathcal E}(t):=V(t)+\int_{\tau=0}^{D(t)} \phi(\tau)\mbox{d}\tau \leq {\mathcal E}_0$, and  fast alignment follows with a finite diameter, $D(t) \leq D_\infty$, dictated by
\[
D(t) \leq D_\infty, \qquad \cM\int_{D_0}^{D_\infty} \phi(s)\mbox{d}s=V_0.
\] 
 Thus, in the case of open space, $\Omega=\R$, compactness of $\{\supr\}$ requires  a finite velocity variation $\displaystyle V_0 <\int_{D_0}^{\infty} \phi(s)\mbox{d}s$. In particular, an \emph{unconditional} flocking follows for global $\phi$'s with unbounded integrable tails. Of course,  in the periodic settings, compactness of the support of $\rho$ is automatic. 

\subsection{Statement of main results}
Lemma \ref{l:V} tells us that  for positive $\phi$'s, the question of flocking is reduced to
the question of global regularity. The latter question --- the global regularity of  \eqref{e:phi}, was addressed in our previous study \cite{ST2016} in the larger context of three  classes of interaction kernels. Namely, for bounded $\phi$'s, global regularity follows for sub-critical initial data such that $u'_0(x)+\phi*\rho_0(x) >0$. For singular kernels $\phi_\alpha(x):=|x|^{-(1+\alpha)}$ corresponding to $\aL_\phi=-(-\Delta)^{\alpha/2}$, global regularity follows for $\alpha \in [1,2)$. Finally, global regularity also holds in the limiting case $\alpha=2$ which corresponds to the Navier-Stokes equations with $\aL_\phi=\Delta$, and we recall that the global regularity for the cases $\alpha\in [1,2]$  is independent of a critical threshold requirement. A main feature of the forcing in all three cases is their commutator structure in \eqref{eq:FH} which yields is the conservative transport of the first-order quantity $u_x+\aL_{\phi}(\rho)$, \cite{CCTT2016,ST2016}
\begin{equation}\label{eq:maine}
e_t+(ue)_x=0, \qquad e:=u_x+\aL_{\phi}(\rho).
\end{equation}

 In this paper, we complement our earlier study of global regularity with the flocking behavior for these classes of interaction kernels. In particular, we make a more precise flocking statement, where fast alignment $\displaystyle \max_{x,y \in \supr}|u(x,t)-u(y,t)| \lesssim e^{-\delta t}$ is strengthened to an exponential decay of slope and curvature of the velocity $|u_x|_\infty +|u_{xx}|_\infty\lesssim e^{-\delta t}$, and the flocking itself is proved in the strong sense of exponential convergence to one of the flocking states $\cF$.
We treat here the flocking behavior in the two cases of positive  $\phi$'s, and of fractional 
$\phi_\alpha,  \ \alpha \in [1,2)$. The  limiting case of Navier-Stokes equations $\aL_\phi=\Delta$ does not seem to satisfy fast alignment due to lack of non-local interactions.  Its large-time behavior remains open.
 
We begin with the case of a bounded positive kernel, and general density \emph{with} a possibility of vacuum. The result is proved in  both periodic and open line domains.
\begin{theorem}[Bounded positive kernels]\label{t:flock-bounded}
Consider the system \eqref{e:phi} with bounded positive kernel $\phi \in W^{2,\infty}(\Omega)$, where $\Omega = \T$ or $\R$. For any initial conditions $(u_0,\rho_0) \in W^{2,\infty}  \times (W^{1,\infty} \cap L^1) $ which satisfies the sub-criticality condition, 
\begin{equation}\label{ }
u'_0+\phi*\rho_0 >0,
\end{equation}
 there exists a unique global solution  $(\rho,u)\in L^\infty([0,\infty); W^{2,\infty}  \times (W^{1,\infty}\cap L^1))$. Moreover, for fixed $\beta <1$ there exists $C,\d >0$ \emph{(}depending on $\beta$\emph{)} such that the velocity  satisfies the fast alignment estimate
\begin{equation}\label{ }
 |u_x(t)|_\infty+ |u_{xx}(t)|_\infty \leq Ce^{-\d t},
\end{equation}
and there is an exponential convergence towards the flocking state $(\bar{u}, \bar{\rho}) \in \cF$, where $\bar{u}=\cP_0/\cM_0$ and  $\bar{\rho}=\rho_\infty(x-t\bar{u}) \in W^{1,\infty}$, 
\begin{equation}\label{ }
 |u(t)-\bar{u}|_{W^{2,\infty}} +|\rho(t) - \bar{\rho}(t) |_{C^\beta} \leq C e^{-\d t}, \qquad t>0.
\end{equation}
\end{theorem}
Next we turn to the case of singular kernels, $\phi_\alpha(x) = |x|^{-(1+\a)}, \ 1\leq \a <2$, in the periodic setting $\Omega=\T$, and no-vacuum condition $\rho_0 >0$. The latter two are necessary to maintain uniform parabolicity of the system.

\begin{theorem}[Singular kernels of fractional order $\alpha\in [1,2)$]\label{t:flock-singular}
Consider the system \eqref{e:phi} with singular kernel $\phi_\alpha(x) =|x|^{-(1+\a)}$, $1\leq \a <2$ on the periodic torus $\T$. For any initial condition $(u_0,\rho_0) \in H^3  \times H^{2+\a}$ away from the vacuum there exists a unique global solution  $(\rho,u)\in L^\infty([0,\infty); H^3  \times H^{2+\a})$. Moreover, for fixed $s<3$ there exists $C,\d >0$ \emph{(}depending on $s$\emph{)} such that the velocity  satisfies the fast alignment estimate,
\begin{equation}\label{e:uxxxT}
 |u_x(t)|_\infty + |u_{xx}(t)|_\infty+|u_{xxx}(t)|_2 \leq Ce^{-\d t},
\end{equation}
and there is an exponential convergence towards the flocking state $(\bar{u}, \bar{\rho}) \in \cF$,   where $\bar{u}=\cP_0/\cM_0$ and  $\bar{\rho}=\rho_\infty(x-t\bar{u})\in H^3$,  
\begin{equation}\label{e:urxxx}
 |u(t)-\bar{u}|_{H^3} +|\rho(t) - \bar{\rho}(t) |_{H^s} \leq C e^{-\d t}, \qquad t>0.
\end{equation}
\end{theorem}

\begin{remark}(On the singular case of fractional order $\alpha \in (0,1)$) We recently learned that  shortly after our release of \cite{ST2016}  another approach to the regularity of \eqref{e:phi} with singular kernels $\phi_\alpha$  appeared in the work of T. Do et. al. \cite{DKRT2017}. Their alternative approach, based on the propagation of properly-tuned modulus of continuity along the lines of \cite{KNV2008}, covers the regularity of \eqref{e:phi}${}_\alpha$ which is treated as critical  system for the range $\alpha\in (0,1)$. Although our arguments of regularity in \cite{ST2016} can be adapted to the respective range, it is not clear whether the flocking statement  in  theorem \thm{t:flock-singular} survives for $0<\a<1$ with either one of these approaches. 
\end{remark}

\section{Flocking with smooth positive kernels}
Our starting point is the conservative transport \eqref{eq:maine}. In the case of positive mollifier we have $[\cL_\phi,u](\rho)= [\phi*,u](\rho)$ and  \eqref{eq:maine} yields
\begin{equation}\label{eq:eeq}
e_t +(ue)_x=0, \qquad e:=u'(x,t)+\phi*\rho(x,t)
\end{equation}
 Hence the positivity  $e_0>0$ propagates in time. It follows that \eqref{e:phi} admits global smooth solutions under the critical threshold condition $e_0(\cdot)>0$, see  \cite{CCTT2016} for details. The flocking of these solutions, which is guaranteed by Lemma~\ref{l:V}, is quantified in the following two lemmas in terms of constants $C,\d >0$, depending on $\cM, |\phi|_{W^{1,\infty}},|u_0|_{W^{1,\infty}}, \min \phi >0$ and $\min e_0 >0$.
\begin{lemma}\label{l:ux-fast} Suppose $e_0 >0$ on $\Omega$. There exist constants $C,\d >0$ such that
\[
|u_x(\cdot,t)|_\infty \leq C e^{-\d t}.
\]
\end{lemma}
\begin{proof}  
We rewrite \eqref{eq:eeq} as a logistic equation along characteristics $\dot{x}(t)=u(x(t),t)$ with non-autonomous threshold $h:=\phi*\rho$,
\begin{equation}\label{ }
\DDt e = (h - e) e, \qquad h=\phi*\rho.
\end{equation}
Here and below, $\displaystyle \DDt$ denotes differentiation along generic particle path $\{x(t)=x(t;x_0)\}$ initiated at $x(t=0;x_0)=x_0$.
Hence, in view of the bound $\cM \iota_\phi \leq h(x,t) \leq \cM I_\phi$,
\[
(\cM I_\phi - e) e \geq \DDt e \geq (\cM \iota_\phi - e) e.
\]
Since $e_0$ is uniformly bounded from above and away from zero, it follows that there exists a  time $t_0>0$ such that the quantity $e(t)=e(t;x_0)$ remains likewise bounded from above and below uniformly for all initial conditions $e_0=e(t=0;x_0)$, 
\begin{equation}\label{e:cC}
C_0:=2 \cM I_\phi \geq e(t)\geq \cM \iota_\phi/2=:c_0>0,\qquad t>t_0.
\end{equation}
 Let us now write the equation for $\phi*\rho$ by convolving the mass equation $\phi*\rho_t + \phi*(\rho u)_x=0$:
 \begin{subequations}\label{eqs:above}
\begin{equation}\label{eq:above}
(\phi*\rho)_t + u (\phi*\rho)_x = -[\phi'*, u]\rho,
\end{equation}
where the commutator is given by
\[
[\phi'*, u]\rho = \int_{\T} \phi'(|x-y|)\big(u(x,t)-u(y,t)\big)\rho(y,t)dy.
\]
In view of \lem{l:V}, and the fact that $\phi \in W^{1,\infty}$ we have
\begin{equation}
|[\phi'*, u]\rho | \leq C\cM |\phi|_{W^{1,\infty}} e^{-\d t}.
\end{equation}
\end{subequations}
In what follows we denote by $E=E(t)$ a generic exponentially decaying quantity, so \eqref{eqs:above} reads $\displaystyle \DDt h = E$, 
uniformly over all initial conditions $x_0\in \Omega$. Let us rewrite this equation for the difference $h = e-u_x$:
\[
\DDt u_x = - u_x e + E.
\]
Recall that for large enough time, $t> t_0$ we have the positive boundedness  $0<c_0<e<C_0$ uniformly over initial conditions. This readily implies
\[
|u_x(x(t;x_0),t)| \leq E(t)
\]
uniformly over $x_0$. Since at any time $t$ characteristics cover all $\Omega$ we arrive at the desired bound $|u_x(t)|_\infty \leq E(t)$.
\end{proof}

Solving the density equation along characteristics we obtain
\begin{equation}\label{ }
\rho(x(t;x_0),t) = \rho_0(x_0) \exp\left\{- \int_0^t u_x(x(s;x_0),s) \ds \right\}.
\end{equation}
So, in view of \lem{l:ux-fast} the density enjoys a pointwise global bound (which is not given a priori)
\begin{equation}\label{eq:rhoi}
\sup_{t>0}| \rho(\cdot,t)|_\infty < \infty.
\end{equation}

Next we establish a second round of estimates in higher order regularity in order to get a control over $\rho_x$ and then prove flocking of the density.

\begin{lemma}\label{l:uxx-fast} There exist $C,\d >0$ such that
\[
|u_{xx}(\cdot, t)|_\infty \leq C e^{-\d t}.
\]
\end{lemma}
\begin{proof}  
Let us write the equations for $e_x$ and $h_x = \phi' * \rho$:
\[
\begin{split}
\DDt e_x & = - u_x e_x+ (h_x - e_x) e \\
\DDt h_x &=   \int_{\T} \phi''(|x-y|)\big(u(x,t)-u(y,t)\big)\rho(y,t)dy.
\end{split}
\] 
Clearly, the right hand side of $h_x$-equation is exponentially decaying, $=E$. Subtracting the two we obtain the equation for $u_{xx}$:
\begin{equation}\label{ }
\DDt u_{xx} = - u_{xx} e  - u_x e_x  +E.
\end{equation}
Note that clearly, $h_x$ is a bounded function, hence  $u_x e_x  =  E u_{xx} + E$. So,
\begin{equation}\label{ }
\DDt u_{xx} = - u_{xx} e  - E u_{xx}  +E.
\end{equation}
Once again, the positive boundedness of $e$ in \eqref{e:cC}, implies that for large enough time,
\begin{equation}\label{ }
\DDt u_{xx} = - u_{xx} e  +E, \qquad t>t_0,
\end{equation}
and the lemma follows. 
\end{proof}
Let's write the equation for $\rho_x$:
\begin{equation}\label{ }
\DDt \rho_x  =  - 2 u_x \rho_x - u_{xx} \rho = E \rho_x + E
\end{equation}
This shows that $\rho_x$ is uniformly bounded. We are now ready to prove \thm{t:flock-bounded}. We state the last computation in a lemma as it will also be used ad verbatim in the next section.

\begin{lemma}\label{l:flock-bdd} Let $(u,\rho) \in W^{2,\infty} \times W^{1,\infty}$ be any solution pair to \eqref{e:phi}.  For fixed $\beta <1$ there exist $C,\d >0$ \emph{(}depending on $\beta$\emph{)} and a flocking pair $(\bar{u}, \bar{\rho}) \in \cF$, $\bar{\rho} \in W^{1,\infty}$, such that,
\begin{equation}\label{ }
 |\rho(\cdot,t) - \bar{\rho}(\cdot,t) |_{C^{\beta}} \leq C e^{-\d t}, \qquad t>0.
\end{equation}
Thus, $\cF$ contains all limiting states of the system \eqref{e:phi}.
\end{lemma}
\begin{proof}
The velocity alignment goes to its natural limit $\bar{u} = \cP/ \cM$.    Denote $\widetilde{\rho}(x,t) := \rho(x+ t \bar{u},t)$. Then $\widetilde{\rho}$ satisfies
\[
\widetilde{\rho}_t + (u - \bar{u}) \widetilde{\rho}_x + u_x \widetilde{\rho} = 0,
\]
where all the $u$'s are evaluated at $x+ t \bar{u}$. According to the established bounds we have $\displaystyle | \widetilde{\rho}_t |_\infty < C e^{-\d t}$.
This proves that $\widetilde{\rho}(t) $ is Cauchy as $t \ra \infty$, and hence there exists a unique limiting state, $\rho_\infty(x)$, such that 
\[
| \widetilde{\rho} (\cdot,t) - \rho_\infty(\cdot)|_\infty < C_1 e^{-\d t}.
\]
Shifting $x$  this  can be expressed in terms of $\rho$ and  $\bar{\rho}(\cdot,t)=\rho_\infty(x-t\bar{u})$
\[
 | \rho(\cdot,t ) - \bar{\rho}(\cdot,t)|_\infty < C_1 e^{-\d t}.
\]
 We also have  $\bar{\rho} \in W^{1,\infty}$ from weak-star compactness. The statement of the lemma follows by interpolation of this  exponential bound with the $C^1$-bound of $\rho$.
\end{proof}

\section{Flocking with singular kernels}
\subsection{Uniform bounds on density, and velocity alignment} The results of this section lead towards the first step in the proof of \thm{t:flock-singular}. However,  we state them in such generality since they hold for a much broader class of kernels satisfying the following three properties.
\begin{itemize}
\item[(i)] Boundedness (away from the origin): for any $r>0$,
\[
 \iota_\phi(r):=\inf_{|x|<r}\phi(|x|) < \infty;
 \]
\item[(ii)] Positivity: $\iota_\phi(2\pi) = \inf_x \phi (|x|) > 0$;
\item[(iii)] Singularity : $\limsup_{r\to 0} r \iota_\phi(r) = \infty$.
\end{itemize}
This class of kernels which was already identified in \cite[Section 3.2]{ST2016},  includes the singular periodized kernels associated with $-(-\partial_{xx})^{\a/2}$, 
\begin{equation}\label{e:phia}
\phi_\a(x) = \sum_{k \in \Z}  \frac{1}{|x+ 2\pi k|^{1+\a}} \ \ \text{ for } 0<\a<2.
\end{equation}
The case of local Laplacian $L = \partial_{xx}$ is not included. 

We recall that due to the positivity (ii), lemma \lem{l:V} applies --- as noted in Example \ref{rem:uk}, fast flocking holds irrespective of (lack of) local integrability 
\begin{equation}\label{ }
V(t) \leq V(0) e^{- \cM \iota_\phi t}.
\end{equation}

As before, our starting point is the conservative transport \eqref{eq:maine} involving $\aL_\a\equiv \aL_{\phi_\a}$,
\begin{equation}\label{ }
e_t + (ue)_x = 0, \qquad e = u_x + \aL_\a( \rho).
\end{equation}
Paired with the mass equation we find that the ratio $q = {e}/{\rho}$ satisfies the transport equation 
\begin{equation}\label{ }
\DDt q = q_t + u q_x = 0.
\end{equation}
Starting from sufficiently smooth initial condition with $\rho_0$ away from vacuum we can assume that 
\begin{equation}\label{ }
Q = |q(t)|_\infty =  | q_0 |_\infty <\infty.
\end{equation}

\begin{lemma}\label{l:r-bounds} Fix $\alpha \in [1,2)$ and consider the pair $(u,\rho)$ which is a smooth solution to \eqref{eq:FH},\eqref{eq:aL}$\mbox{}_\a$ subject to initial density $\rho_0$ away from vacuum. Then there are positive constants $c = c(\cM,Q,\phi)$, $C = C(\cM,Q,\phi)$ such that 
\begin{equation}\label{ }
0< c \leq \rho(x,t) \leq C, \qquad  x\in\T, \ t \geq 0.
\end{equation}
\end{lemma}
\begin{remark}
The upper bound and the weaker lower-bound $\rho(\cdot,t)\gtrsim 1/t$ was established for the CS hydrodynamics \eqref{eq:CS} along the line of the 1D decay for commutator forcing \cite{ST2016}. We include both cases for completeness.
\end{remark}

\begin{proof}  Let us recall that the density equation can be rewritten as
\begin{equation}\label{ }
\rho_t + u \rho_x = - q \rho^2 + \rho \aL_\a(\rho).
\end{equation}
Let us evaluate at a point where the maximum $\rmax$ is reached at $\xmax$. We obtain
\[
\begin{split}
\ddt \rmax & = - q(\xmax,t) \rmax^2 + \rmax \int \phi(|z|) ( \rho(\xmax + z,t) - \rmax) \dz \\
& \leq Q \rmax^2 +  \rmax \int_{|z| <r} \phi(|z|) ( \rho(\xmax + z,t) - \rmax) \dz \\
& \leq Q \rmax^2 + \iota_\phi(r)  \rmax  (  \cM - r \rmax) = Q \rmax^2 + \iota_\phi(r) \cM \rmax - r \iota_\phi(r) \rmax^2.
\end{split}
\]
In view of assumption (iii) we can pick  $r$  large enough to satisfy $r \iota_\phi(r) > Q+1$, while according to (i) $ \iota_\phi(r)$ itself remains finite. We thus achieve inequality
\[
\ddt \rmax \leq -  \rmax^2 +  \iota_\phi(r) \cM \rmax,
\]
which establishes the upper bound by integration. 

As to the lower bound we argue similarly. Let $\rmin$ and $\xmin$ the minimum value of $\rho$ and a point where such value is achieved. We have
\[
\begin{split}
\ddt \rmin & \geq -Q \rmin^2 + \rmin \int_\T  \phi(|z|) ( \rho(\xmin + z,t) - \rmin) \dz\\
& \geq   -Q \rmin^2  +  \iota_\phi(2\pi)\rmin ( \cM - 2\pi \rmin) = -(Q +2\pi \iota_\phi(2\pi)) \rmin^2  + \iota_\phi(2\pi) M \rmin.
\end{split}
\]
In view of (ii) the linear term on the right hand side has a positive coefficient. This readily implies the uniform lower bound.
\end{proof}

Subsequently we focus solely on the critical case $\a=1$. The case $1<\a<2$ requires the same prerequisites as established in this section, however, incorporating it into the existing proof below would actually require  fewer steps due to excess of dissipation. We will skip those for the sake of brevity of what will already be a technical exposition.

\subsection{An initial approach towards \thm{t:flock-singular}} 
In this section we specialize on the case of critical singular kernel, and establish uniform control on the first order quantities $|\rho_x|_\infty$, $|u_x|_\infty$. We assume that we start with initial condition $(u_0,\rho_0) \in H^3(\T)$.  As shown in \cite{ST2016} there exists a global solution to \eqref{eq:FH} in the same space. So, we can perform all of the computations below as classical. 

First, let us recite one argument from \cite{ST2016}.  
Recall that the density $\rho$ satisfies the following parabolic form  of the density equation, expressed in terms of $\aL_1\equiv -\L$, 
\begin{equation}
\rho_t+u\rho_x+e\rho= -\rho \L(\rho), , \qquad \L(f)=p.v. \int_\R \frac{f(x)-f(y)}{|x-y|^2}dy.
\end{equation}
Similarly, one can write the equation for the momentum $m = \rho u$:
\begin{equation}
m_t+u m_x+e m  = -\rho \L(m).
\end{equation}
In both cases the drift $u$ and the forcing $e \rho$ or $e m$ are bounded a priori due to the maximum principle and \lem{l:r-bounds}. Moreover, the diffusion operator has kernel
\[
K(x,h,t) = \rho(x) \frac{1}{|h|^2}.
\]
Using lower bound on the density from \lem{l:r-bounds} we conclude that the kernel falls under the assumptions of Schwab and Silverstre \cite{SS2016} uniformly on the time line. A direct application of \cite{SS2016} implies that there exists an $\g>0$ such that
\begin{equation}\label{e:gamma}
\begin{split}
| \rho|_{C^\g(\T \times [t+1,t+2))} &\leq C( |\r|_{L^\infty(t,t+2)} + |\rho e |_{L^\infty(t,t+2)})\\
| m |_{C^\g(\T \times [t+1,t+2))} &\leq C( |m|_{L^\infty(t,t+2)} + |me|_{L^\infty(t,t+2)}) \\
| u |_{C^\g(\T \times [t+1,t+2))} &\leq C( |u|_{L^\infty(t,t+2)}, |\rho|_{L^\infty(t,t+2)}) ,
\end{split}
\end{equation}
holds for all $t>0$. Since the right hand sides are uniformly bounded on the entire line we have obtained uniform bounds on $C^\g$-norm starting from time $t=1$. Since we are concerned with long time dynamics let us reset initial time to $t=1$, and allow ourselves to assume that $C^\g$-norms are bounded from time $t=0$.

\begin{lemma}\label{l:rslope} We make the same assumptions stated in Theorem \ref{t:flock-singular}. Then the following  uniform bound holds
\begin{equation}\label{ }
\sup_{t \ge 0 }  |\rho_x(\cdot,t) |_\infty  <  \infty.
\end{equation}
\end{lemma}
\begin{proof} The argument goes  verbatim as presented in \cite[Section 6.2]{ST2016}  with all the constants involved being uniform in view of the established bounds above.  We recall the penultimate inequality
\[
\ddt |\rho'|^2 \leq c_1 + c_2 |\rho'|^2 - c_3 \dD \rho'(x),
\]
where all the quantities are evaluated at a point of maximum of $|\r'|$, and where
\[
\dD\rho'(x) = \int_\R  \frac{|\rho'(x) - \rho'(x+z)|^2}{|z|^2} dz.
\]
Using the nonlinear maximum bound from \cite{CV2012}
\[
\dD\rho'(x) \geq c_4 \frac{|\rho'|_\infty^3}{|\rho|_\infty} \geq c_5 |\rho'|_\infty^3,
\]
we can further hide the quadratic term into dissipation to obtain
\begin{equation}\label{e:rho-derb}
\ddt |\rho'|^2 \leq c_6 - c_7 \dD \rho'(x) \leq c_6 - c_8 |\rho'|^3.
\end{equation}
This enables us to conclude the Lemma. 
\end{proof}

We turn to study the flocking behavior in the singular case.
To this end, we first prepare preliminary estimates on the dissipation terms to be encountered in the sequel. Below is an improvement on nonlinear maximum principle bound of \cite{CV2012} in the case of small amplitudes. As a byproduct we obtain a trilinear estimate that will be used in the sequel. 

\begin{lemma}[Enhancement of dissipation by small amplitudes] \label{l:enhance} Let $u \in C^1(\T)$ be a given function with amplitude $V = \umax - \umin$. There is an absolute constant $c_1 >0$ such that the following pointwise estimate holds
\begin{equation}\label{eq:x1}
\dD u'(x) \geq c_1 \frac{|u'(x)|^3}{V}.
\end{equation}
In addition, there is an absolute constant $c_2>0$ such that for all $B>0$ one has
\begin{equation}\label{eq:x2}
\dD u'(x) \geq B |u'(x)|^2 - c_2 B^3 V^2.
\end{equation}
\end{lemma}
\begin{proof} We start as in \cite{CV2012}. Using smooth truncations in the integrals we obtain, for every $r>0$:
\[
\dD u'(x) \geq \int_{|z|>r} \frac{ |u'(x)|^2 - 2 u'(x+z) u'(x) }{|z|^2} \dz,
\]
where we dropped the positive term with $|u'(x+z)|^2$.  Now, using that $u'(x+z) = u_z(x+z) = (u(x+z) - u(x))_z$ we can integrate by parts in the second term to obtain
\[
\dD u'(x) \geq  \frac{|u'(x)|^2}{r} + 4 u'(x) \int_{|z|>r} \frac{u(x+z) - u(x)}{|z|^4} z \dz \geq  \frac{|u'(x)|^2}{r} - c_0 |u'(x)| V \frac{1}{r^2}.
\]
By picking $r = \frac{2c_0 V}{|u'(x)|}$ we obtain \eqref{eq:x1}. Picking $r = \frac{1}{2B}$ and using Young's inequality ,
\[
\dD u'(x) \geq 2B |u'(x)|^2 -4c_0 |u'(x)|V B^2 \geq B |u'(x)|^2 - 16c_0B^3 V^2,
\]
and \eqref{eq:x2} follows.
\end{proof}

We proceed to the exponential decay of $u_x$ and $u_{xx}$  which is quantified  in the next two lemmas, in terms of constants $C,\d >0$, depending on $\cM, |u_0|_{H^3}$ and additional parameters specified below.
\begin{lemma} \label{l:vanslope} We make the same assumptions stated in Theorem \ref{t:flock-singular}. There exist constants $C, \d>0$ such that for all $t>0$ one has
\begin{equation}\label{ }
|u_x(\cdot,t) |_\infty \leq Ce^{-\d t}.
\end{equation}
\end{lemma}
\begin{proof}
Differentiating the $u$-equation and evaluating at a point of maximum we obtain
\[
\ddt |u'|^2 \leq |u'|^3 + \aT(\rho', u) u'  + \aT(\rho, u') u', \qquad \aT(\rho,u):=-\L(\rho u)+u\L(\rho).
\]
First, as to the dissipation term, let us fist observe
\[
(u'(y) - u'(x))u'(x)  = - \frac12 |u(y) - u(x)|^2 + \frac12( |u'(y)|^2 - |u'(x)|^2 )\le - \frac12 |u(y) - u(x)|^2.
\]
Thus, in view of \lem{l:r-bounds}, 
\[
\aT(\rho, u') u'  (x) \le - c_1 \dD u'(x).
\]
Let us turn to the remaining term $\aT(\rho', u) u' $.   We have
\[
\aT(\rho', u) u' = |u'(x) |^2 \L \rho + u'(x) \int \rho'(x+z) \frac{\d_z u(x) - z u'(x)}{|z|^2} \dz.
\]
Expressing $\L \rho = e- u_x$ the first term, and using uniform bound on $\rho_x$ from \lem{l:rslope} we obtain the bound  
\[
| \aT(\rho', u) u'  | \leq c_2 |u'|^3 + c_3 |u'|^2 + c_4 |u'| \left| \int \rho'(x+z)\frac{ \d_z u(x) - z u'(x)}{|z|^2} \dz \right|.
\]
To bound the integral we split it into the long range $\{|z| >\pi\}$ and short range $\{|z| \leq \pi\}$ parts. For the  short range  we use the bound
\[
| \d_z u(x) - z u'(x) |\leq  \int_0^z |u'(x+w) - u'(x)| \dw \leq |z|^{3/2} \dD^{1/2} u'(x),
\]
So,
\[
\left| \int_{|z|\leq \pi} \rho'(x+z)\frac{ \d_z u(x) - z u'(x)}{|z|^2} \dz\right| \leq |\rho'|_\infty  \dD^{1/2} u'(x) \leq C\dD^{1/2} u'(x).
\]
For the long range part, we apply the following argument (to be used several other times in the sequel). We have
\begin{equation}\label{e:triple1}
\begin{split}
&\int_{|z| \geq \pi}  \frac{\rho'(x+z)\d_zu(x)}{|z|^2} \dz =\sum_{k\neq 0} \int_{|z|\leq \pi}  \frac{\rho'(x+z)\d_zu(x)}{|z + 2\pi k |^2} \dz \\
&\leq c_6 \sum_{k\neq 0} \frac{1}{k^2} \int_{|z|\leq \pi}  |\rho'(x+z)||z||u'|_\infty \dz \leq c_7 |u'|_\infty |\rho'|_1 \leq C |u'|_\infty. 
\end{split}
\end{equation}
Proceeding to the second part,
\begin{equation}\label{e:triple2}
\begin{split}
& \int_{|z| \geq \pi}  \frac{\rho'(x+z) u'(x)}{z} \dz = \sum_{k\neq 0} \int_{|z|\leq \pi}  \frac{\rho'(x+z)u'(x)}{z + 2\pi k} \dz \\
&=  \sum_{k>0} \int_{|z|\leq \pi}  \rho'(x+z)u'(x) \frac{2z}{|z|^2 - 4\pi^2 k^2} \dz \leq |u'|_\infty |\rho'|_1 \leq C|u'|_\infty.
 \end{split}
\end{equation}
Putting the estimates together we obtain
\[
| \aT(\rho', u) u'  |  \leq c_2 |u'|^3 + c_4 |u'|^2 + c_5 |u'| \dD^{1/2} u'(x) \leq c_2 |u'|^3 + c_6 |u'|^2  + \frac12 c_1\dD u'(x).
\]
Altogether we have obtained, resetting the constant counter,
\[
\ddt |u'|^2  \leq  c_2 |u'|^3 + c_6 |u'|^2 - c_7 \dD u'(x).
\]
In view of \lem{l:enhance}, we have 
\[
\ddt |u'|^2  \leq c_8 V -  c_9 |u'|_\infty^2,
\]
Integrating we discover that $|u'| \to 0$ at an exponential rate of at least $\frac12 \min\{c_{9}, cI_\phi\}$, the latter being the rate of decay of $V$. This finishes the proof.
\end{proof}

The estimates established so far are sufficient to prove an initial version of \thm{t:flock-singular} as stated in \lem{l:flock-bdd}. The proof the lemma goes ad verbatim in the present case.  

\subsection{Completing the proof of \thm{t:flock-singular}}  In  this section we perform  computations with the goal to show that the flocking proved in  \lem{l:flock-bdd} takes place in all spaces up to $H^3$, and that the limiting profile $\bar{\rho}$ itself belongs to $H^3$.  We prove exponential flattening of $u$ in terms of curvature $|u''|_\infty$ and third derivative $u'''$ in $L^2$. This complements the statement of  \lem{l:flock-bdd} to the full extent of \thm{t:flock-singular}. 

We start by showing exponential decay of $|u''|_\infty$. As before we denote by $E= E(t)$ any quantity with an exponential decay, e.g. $|u'|_\infty =E$, or $V = E$.  Thus, according to \lem{l:enhance}, we have pointwise bounds
\begin{equation}\label{e:Duxx}
\begin{split}
\dD u''(x) & \geq \frac{|u''(x)|^3}{E},\\
\dD u''(x) & \geq B |u''(x)|^2 - C(B) E.
\end{split}
\end{equation}
Due to these bounds the dissipation term absorbs all cubic and quadratic terms with bounded coefficients. It does the latter at the cost of adding a free $E$-term unattached to either $u$ or $\rho$. Finally, let us recall from \cite{ST2016} that the quantity $Q= (e/\rho)_x / \rho$ is transported:  $ Q_t + u Q_x = 0$. As such, it remains bounded uniformly for all times. Expressing $e'$ from $Q$, we see that $e'$ is controlled by $\rho'$ and $\rho$, which in view of \lem{l:rslope} implies uniform bound on $e'$:
\begin{equation}\label{e:ex}
\sup_t |e'(\cdot, t)|_\infty <\infty.
\end{equation}
Using this additional piece of information we are in a position to prove control of the curvature.

\begin{lemma} \label{l:vanslope2} We make the same assumptions stated in Theorem \ref{t:flock-singular}. There are constants $C, \d>0$ such that for all $t>0$ one has
\begin{equation}\label{ }
|u_{xx}(\cdot, t) |_\infty \leq Ce^{-\d t}.
\end{equation}
\end{lemma}
\begin{proof} Evaluating the $u$-equation at a point of maximum and performing the same initial steps as in \lem{l:vanslope} we obtain
\begin{equation}\label{e:uxx}
\ddt |u''|^2 \leq E |u''|^2 - c_0 \dD u''(x) + \aT(\r'', u) u'' + 2 \aT(\rho',u')u''.
\end{equation}
We have
\[
\aT(\r'', u) u''(x) =  \int_\R  \frac{\rho''(x+z)( \d_zu(x) - z u'(x)) u''(x)}{|z|^2} \dz+ \L \rho'(x) u'(x) u''(x).
\]
For the $\L \rho'$ term, in view of \eqref{e:ex}, we argue that $|\L \rho'| = |e' - u'' | \leq c_1 + |u''|$. Thus,
\begin{equation}\label{ }
| \L \rho'(x) u'(x) u''(x) | \leq E (|u''| + |u''|^2).
\end{equation}
As to the integral term, first, we handle the short range part as usual: 
\[
 \int_{|z| \leq \pi}  \frac{\rho''(x+z)( \d_zu(x) - z u'(x)) u''(x)}{|z|^2} \dz \leq |u''|_\infty^2 \int_{|z| \leq \pi} | \rho''(x+z)| \dz  \leq  |u''|_\infty^2 |\rho''|_1.
\]
However, note that $|\rho''|_1 \leq |\rho''|_2 \leq |\L \rho'|_2 \leq |e'|_2 + |u''|_2 \leq c_2 +c_3 |u''|_\infty$. Putting all estimates together we obtain
\begin{equation}\label{ }
|\aT(\r'', u) u'' | \leq E |u''| + c_4 |u''|^2 + c_5 |u''|^3  +\int_{|z| \geq \pi}  \frac{\rho''(x+z)( \d_zu(x) - z u'(x)) u''(x)}{|z|^2} \dz .
\end{equation}
As the for long range integral extra care is needed due to periodicity of functions, and we have to avoid having first degree term $|u''|$ appearing without exponentially decaying weight.  So, performing exactly the same computation as in \eqref{e:triple1} - \eqref{e:triple2}, with $\rho'$ replaced by $\rho''$ we get
\[
\left| \int_{|z| \geq \pi}  \frac{\rho''(x+z)(\d_zu(x)+ zu'(x))}{|z|^2} \dz \right| \leq  c_6 |u'|_\infty |\rho''|_1 \leq  E (1+ |u''|_\infty).
\]
Collecting the estimates we obtain
\[
|\aT(\r'', u) u'' | \leq E |u''| + c_8 |u''|^2 + c_9 |u''|^3.
\]
To bound the remaining term $ \aT(\rho',u')u''$ we will make use of the dissipation. 
\[
\aT(\r', u') u'' =  \int_\R  \frac{\rho'(x+z)( \d_zu'(x) - z u''(x)) u''(x)}{|z|^2} \dz+ \L \rho(x) |u''(x)|^2.
\]
We have
\[
| \L \rho(x) | |u''(x)|^2 = |e - u'| |u''|^2 \leq c |u''|^2.
\]
For the small scale part we have
\[
\begin{split}
\left| \int_{|z|<\pi}  \frac{\rho'(x+z)( \d_zu'(x) - z u''(x)) u''(x)}{|z|^2} \dz  \right| &\leq |u''|  \int_{|z|<\pi}  \frac{|\rho'|_\infty \dD^{1/2} u''(x) }{|z|^{1/2}} \dz \\
& \leq  c_{10} |u''| \dD^{1/2} u''(x) \leq c_{11} |u''|^2 + \frac{c_0}{4}\dD u''(x).
\end{split}
\]
For the large scale part we have
\[
\left| \int_{|z|\geq \pi}  \frac{\rho'(x+z)( \d_zu'(x) - z u''(x)) u''(x)}{|z|^2} \dz\right| \leq |\rho'||u''|^2.
\]
Thus,
\[
 |\aT(\r', u') u'' | \leq E | u'' |_\infty+ c_{14} | u'' |_\infty^2 + \frac{c_0}{4}\dD u''(x).
\]
Gathering the obtained estimates into  \eqref{e:uxx} we obtain
\begin{equation}\label{e:uxx2}
\ddt |u''|^2 \leq E | u'' |_\infty+ c_{15} | u'' |_\infty^2+c_{16} | u'' |_\infty^3 - c_{17} \dD u''(x).\end{equation}
Furthermore, $E | u'' |_\infty \lesssim E^2 + |u''|^2_\infty$. In view of \eqref{e:Duxx} the dissipation term absorbs the quadratic and cubic terms, and we are left with
\begin{equation}\label{ }
\ddt |u''|^2 \leq E - c_{18} \dD u''(x) \lesssim E -  |u''(x) |^2.
\end{equation}
This finishes the proof.
\end{proof}

\begin{corollary}\label{c:rxx}
We have, for every $1\leq p < \infty$,
\[
\sup_t  |H\rho''(\cdot, t)|_\infty <\infty.
\]
\end{corollary}
We have $|H\rho''|_\infty = |\L \rho '|_\infty \leq |e'|_\infty + |u''|_\infty$. So, the corllary simply follows from \lem{l:vanslope2}  and \eqref{e:ex}.  Since we are in the torus settings, this automatically implies uniform bound for all $L^p$-norms of $\rho''$, for $1\leq p < \infty$:
\begin{equation}\label{e:rLp}
\sup_t |\rho''(\cdot,t)|_p <\infty.
\end{equation}
In what follows we tacitly use these bounds by simply replacing uniformly bounded quantity such as above by constants. 

We are now in a position to perform final estimates in the top regularity class $H^3$. 

\begin{lemma} \label{l:vanslope3} We make the same assumptions stated in Theorem \ref{t:flock-singular}. There are constants $C, \d>0$ such that for all $t>0$ one has
\begin{equation}\label{ }
\begin{split}
|u_{xxx}(\cdot,t) |_2 & \leq Ce^{-\d t} \\
|\rho_{xxx}(\cdot,t)|_2 & \leq C.
\end{split}
\end{equation}
\end{lemma}

First we need a universal bound on the large scale of a triple product, similar in spirit to \eqref{e:triple1}-\eqref{e:triple2} which we recast more generally to suit the context of \lem{l:vanslope3}.

\begin{lemma}\label{l:large} For any three $2\pi$-periodic function $f, g, h$ we have the following bound
\begin{equation}\label{e:large}
\left| \int_\T \int_{|z|\geq \pi}  \frac{f(x+z)( \d_z g (x) - z g'(x)) h(x)}{|z|^2} \dz \dx \right | \leq C  |f|_{p_1} |g'|_{p_2} |h|_{p_3}, 
\end{equation}
for any conjugate triple $\frac{1}{p_1} + \frac{1}{p_2} + \frac{1}{p_3} = 1$. 
\end{lemma}
\begin{proof}
\[
\begin{split}
&\int_\T \int_{|z|\geq \pi}  \frac{f(x+z)( \d_z g (x) - z g'(x)) h(x)}{|z|^2} \dz \dx = \sum_{k \neq 0}  \int_\T \int_{|z|\leq \pi}  \frac{f(x+z) \d_z g (x)}{|z + 2\pi k|^2} \dz h(x) \dx \\
&+ \int_\T \int_{|z|\leq \pi}  \frac{f(x+z)}{z + 2\pi k} \dz\, g' (x) h(x) \dx \\
&\leq  \int_0^1  \int_{\T^2}  | f(x+z) g' (x+\th z) h(x) |  \dx \dz \dth  + \int_{\T} \int_{|z|\leq \pi}  f(x+z) g' (x) h(x) \sum_{k>0} \frac{2z}{|z|^2 - 4\pi^2 k^2} \dx \dz \\
&\leq  \int_0^1  \int_{\T^2}  | f(x+z) g' (x+\th z) h(x) |  \dx \dz \dth  + \int_{\T^2} | f(x+z) g' (x) h(x) |\dx \dz \\
&\leq C |f|_{p_1} |g'|_{p_2} |h|_{p_3} 
\end{split}
\]
for any conjugate triple $\frac{1}{p_1} + \frac{1}{p_2} + \frac{1}{p_3} = 1$. 
\end{proof}

\begin{proof}[Proof of \lem{l:vanslope3}]
Once we establish exponential decay of $|u'''|_2$, it would imply control over $|\rho'''|_2$ via $e$ as follows. Note that $e''$ satisfies
\[
\ddt e'' + u e''' + 2u' e'' + 2 u'' e' + u''' e = 0.
\]
Testing with $e''$ we obtain
\begin{equation}\label{e:exx}
\ddt |e''|_2^2 \leq  3 u' e''e'' + 2 u'' e' e'' + u''' e e'' \leq E(|e''|_2^2 + |e''|_2).
\end{equation}
This readily implies global uniform bound on $|e''|_2$, and hence on $|\rho'''|_2$.

Let us write the equation for $u'''$:
\begin{equation}\label{e:u'''}
u'''_t + u u'''_x + 4u'u'''+ 3u'' u'' = \aT(\r''',u)+ 3 \aT(\r'',u') + 3\aT(\r',u'') + \aT(\r,u''').
\end{equation}
Testing with $u'''$ we obtain (we suppress integral signs and note that $\int u''u''u''' = 0$)
\begin{equation}\label{e:eb}
\begin{split}
\ddt |u'''|_2^2 &= - 7 u'(u''')^2 + 2 \aT(\r''',u)u''' +6 \aT(\r'',u')u''' + 6 \aT(\r',u''))u''' + 2 \aT(\r, u''')u''' \\
& \leq E |u'''|_2^2 - c_0  \int \dD u''' \dx+2 \aT(\r''',u)u''' +6 \aT(\r'',u')u''' + 6 \aT(\r',u'')u'''.
\end{split}
\end{equation}
Note that $ \int \dD u''' \dx = |u'''|_{H^{1/2}}^2$.   As follows from \lem{l:enhance} we have the lower bound
\begin{equation}\label{e:DuxxxInt}
\int_\T \dD u''' \dx  \geq B |u'''|_2^2 - C(B) E, \text{ for any } B>0.
\end{equation}
Again, the dissipation absorbs all quadratic terms. Let us note that we cannot rely on the pointwise inequality $|e''| \lesssim |\rho''|$ since it requires regularity higher than $H^3$. Hence, the argument has to be genuinely $L^2$ based. We also point out that the argument of \cite{ST2016} is rough for the purposes of long time asymptotics.

 We have
\[
| \aT(\r''',u)u''' | = \int H \rho''' u' u''' \dx + \int_{\T^2} \frac{\rho'''(x+z)( \d_z u (x) - z u'(x)) u'''(x)}{|z|^2} \dz \dx.
\]
Clearly, $ | \int H \rho''' u' u''' \dx | \leq E |\rho'''|_2 |u'''|_2$. In view of \eqref{e:large}, the last integral in the range $|z| \geq \pi$ is bounded by the same $|\rho'''|_2 |u'''|_2 |u'|_\infty \leq E|\rho'''|_2 |u'''|_2$. In the range  $|z| \leq \pi$ we simply use $| \d_z u (x) - z u'(x) | \leq |z|^2 |u''|_\infty$. Thus, this part is also bounded by $E|\rho'''|_2 |u'''|_2$. We have proved 
\[
| \aT(\r''',u)u''' | \leq E|\rho'''|_2 |u'''|_2.
\]
Next,
\[
\aT(\r'',u')u''' =  \int_\T H \rho'' u'' u''' \dx + \int_{\T^2} \frac{\rho''(x+z)( \d_z u' (x) - z u''(x)) u'''(x)}{|z|^2} \dz \dx.
\]
In view of \cor{c:rxx}, $ \int_\T H \rho'' u'' u''' \dx  \leq E |u'''|_2 \leq E^2 + |u'''|_2^2$. Using \eqref{e:large}, we estimate the large scale of the integral by $|\rho''|_2 |u''|_\infty |u'''|_2 \leq E |u'''|_2$.  As to the small scale, we first observe
\[
| \d_z u' (x) - z u''(x) | = \left| \int_0^z (u''(x+w) - u''(x) ) \dx \right| \leq  \left( \int_0^z \frac{|u''(x+w) - u''(x) |^4}{|w|^4} \dw \right)^{1/4} |z|^{7/4}.
\]
Thus, 
\[
\begin{split}
&\left|  \int_{\T} \int_{|z|< \pi}  \frac{\rho''(x+z)( \d_z u' (x) - z u''(x)) u'''(x)}{|z|^2} \dz \dx \right| \\
&\leq  \int_{\T} \int_{|z|< \pi}  |\rho''(x+z)|\left( \int \frac{|u''(x+w) - u''(x) |^4}{|w|^4} \dw \right)^{1/4}| u'''(x)|  \dx |z|^{-1/4} \dz \\
&\leq |\rho''|_4 |u''|_{W^{3/4,4}} |u'''|_2 \leq  C |u'''|_{H^{1/2}}^{1/2} |u''|_\infty^{1/2} |u'''|_2 \leq E^4 + c_1  |u'''|_2^2 + \frac12 c_0 |u'''|_{H^{1/2}}^2,
\end{split}
\]
where in the last steps we used  Gagliardo-Nirenberg inequality and \cor{c:rxx}. All in all, we obtain
\[
| \aT(\r'',u')u''' | \leq E + c_2  |u'''|_2^2 + \frac14 c_0 |u'''|_{H^{1/2}}^2.
\]
Lastly, in the remaining the term $\aT(\r',u'')u'''$ we make one preparatory step in which we first move one derivative from $u$'s over onto $\rho'$.  To this end, we use symmetrization as follows
\begin{equation*}\label{}
\begin{split}
\aT(\r',u'')u''' & = \int \r'(y) u'''(x) (u''(y) - u''(x)) \frac{\dy \dx}{|x-y|^2} \\
&= \frac12 \iint (\r'(y) u'''(x) -\r'(x) u'''(y)) (u''(y) - u''(x)) \frac{\dy\dx}{|x-y|^2} \\
& = \frac12 \iint (\r'(y)-\r'(x)) u'''(x) (u''(y) - u''(x)) \frac{\dy \dx}{|x-y|^2} \\
&+ \frac12 \iint \r'(x) (u'''(x)-u'''(y))(u''(y) - u''(x)) \frac{\dy \dx}{|x-y|^2}\\
& =  \frac12 \iint \d_z \r'(x) u'''(x) \d_z u''(x) \frac{\dz \dx}{|z|^2} + \frac12 \iint \r'(x) \d_z u'''(x) \d_z u''(x) \frac{\dz \dx}{|z|^2}.
\end{split}
\end{equation*}
Thus, in the second term we have a full derivative $\d_z u'''(x) \d_z u''(x) = (( \d_z u''(x) )^2)'$. So,  integrating by parts, we obtain
\[
\begin{split}
\iint \r'(x) \d_z u'''(x) \d_z u''(x) \frac{\dz \dx}{|z|^2} &= - \frac12 \iint \r''(x) | \d_z u''(x) |^2\frac{\dz \dx}{|z|^2}  \leq c_3 |\rho''|_2 |u''|^2_{W^{3/4,4}}\\
& \leq  c_4 |u'''|_{H^{1/2}} |u''|_\infty \leq E +  \frac14 c_0 |u'''|_{H^{1/2}}^2.
\end{split}
\]
In the first term, we estimate, by \eqref{e:large},
\[
\begin{split}
\iint \d_z \r'(x) u'''(x) \d_z u''(x) \frac{\dz \dx}{|z|^2} & =\int \L \rho |u'''|^2 +  \iint \d_z \r'(x) u'''(x)( \d_z u''(x)- zu'''(x)) \frac{\dz \dx}{|z|^2} \\
& \leq |\L \r|_\infty |u'''|_2^2 + |\r'|_\infty |u'''|_2^2 \\
&+  \iint_{|z|<\pi} \d_z \r'(x) u'''(x)( \d_z u''(x)- zu'''(x)) \frac{\dz \dx}{|z|^2} \\
& \leq c_5 |u'''|_2^2 + |u'''|_{H^{1/2}} |u'''|_2 \leq c_6  |u'''|_2^2 + \frac14 c_0  |u'''|^2_{H^{1/2}}.
\end{split}
\]
Thus,
\[
| \aT(\r',u'')u''' | \leq E+ c_7|u'''|_2^2 + \frac12 c_0   |u'''|^2_{H^{1/2}}.
\]
In view of \eqref{e:DuxxxInt} the dissipation term absorbs all quadratic terms, and we arrive at 
\begin{equation}\label{e:uxxx}
\ddt |u'''|_2^2 \leq C(B) E - B |u'''|_2^2  + E|\rho'''|^2_2.
\end{equation}
Extra care is needed due to the last term since we don't know yet how fast $|\r'''|_2$ can grow. Let us get back to the ``e" term.  As before we have
\begin{equation}\label{e:exx2}
\begin{split}
\ddt |e''|_2^2 & \leq  3 u' e''e'' + 2 u'' e' e'' + u''' e e'' \leq E |e''|_2^2 + E + |u'''|_2 |e''|_2\\
& \leq E |e''|_2^2 + E + C(\e)|u'''|^2_2 + \e |e''|^2_2,
\end{split}
\end{equation}
for every $\e>0$.  Fix an arbitrarily small $\e>0$, and a pick large $B > 4 C(\e)$. Add the two equations \eqref{e:uxxx}, \eqref{e:exx2} together. Noting that $X = |u'''|_2^2 + |\rho'''|_2^2 \sim  |u'''|_2^2 +|e''|_2^2 \sim  |\rho'''|_2^2 +|e''|_2^2$, we obtain
\[
\ddt X \leq C(B) E + E X + \e X \lesssim  E+ \e X.
\]
This shows that $X$, and in particular $|\rho'''|_2$, grows at an arbitrarily small exponential rate $\e$. Using it back into \eqref{e:uxxx}, we see that in the product $E|\rho'''|^2_2$  the rate of exponential decay of $E$ is fixed and positive, yet that of $|\rho'''|_2$ is arbitrarily small. Hence the product decays exponentially, and we arrive at
\begin{equation}\label{e:uxxx2}
\ddt |u'''|_2^2 \leq E - B |u'''|_2^2.
\end{equation}
This proves the lemma.
\end{proof}

As a consequence we readily obtain the full statement of \thm{t:flock-singular}. Namely, \eqref{e:uxxxT} follows directly from \lem{l:vanslope3}, and the convergence for densities stated in \eqref{e:urxxx} follows by interpolation between exponential decay in $L^\infty$ and uniform boundedness in $H^3$. The fact that $\bar{\rho} \in H^3$ is simple consequence of uniform boundedness of $\rho(t)$ in $H^3$ and weak compactness.


\begin{thebibliography}{11}

\bibitem[CCTT2016]{CCTT2016}
Jos{\'e}~A. Carrillo, Young-Pil Choi, Eitan Tadmor, and Changhui Tan.
\newblock Critical thresholds in 1{D} {E}uler equations with non-local forces.
\newblock {\em Math. Models Methods Appl. Sci.}, 26(1):185--206, 2016.

\bibitem[CCP2017]{CCP2017}
J. Carrillo, Y.-P. Choi, and S. Perez,
A review on attractive-repulsive hydrodynamics
for consensus in collective behavior,
in ``Active Particles,  Volume 1. 
Advances in Theory, Models, and Applications''
(N. Bellomo, P. Degond and E. Tadmor, eds.),
Birkhäuser 2017.


\bibitem[CV2012]{CV2012}
Peter Constantin and Vlad Vicol.
\newblock Nonlinear maximum principles for dissipative linear nonlocal
  operators and applications.
\newblock {\em Geom. Funct. Anal.}, 22(5):1289--1321, 2012.


\bibitem[DKRT2017]{DKRT2017}
Tam Do, Alexander Kiselev, Lenya Ryzhik, and Changhui Tan
\newblock Global regularity for the fractional Euler alignment system,  arXiv:1701.05155.

\bibitem[HT2008]{HT2008}
S.-Y. Ha and E. Tadmor
{From particle to kinetic and hydrodynamic descriptions of flocking},
Kinetic and Related Models, 1, no. 3, (2008), 415--435.


\bibitem[ISV2016]{ISV2016}
C.~Imbert, R.~Shvydkoy, and F.~Vigneron.
\newblock Global well-posedness of a non-local {B}urgers equation: The periodic
  case.
\newblock to appear in {\em Annales math\'ematiques de Toulouse}, 2016.



\bibitem[KNV2008]{KNV2008}
Kiselev, A. and Nazarov, F. and Volberg, A.
\newblock Global well-posedness for the critical 2{D} dissipative
              quasi-geostrophic equation
 \newblock Invent. Math., 167(3) (2007) 445-453.
	

\bibitem[MT2011]{MT2011}
S. Motsch and E. Tadmor
\newblock A new model for self-organized dynamics and its flocking behavior, 
\newblock J. Stat. Physics 144(5) (2011) 923-947.

\bibitem[MT2014]{MT2014}
S. Motsch and E. Tadmor
\newblock Heterophilious dynamics enhances consensus.
\newblock SIAM Review 56(4) (2014) 577--621.


\bibitem[SS2016]{SS2016}
Russell~W. Schwab and Luis Silvestre.
\newblock Regularity for parabolic integro-differential equations with very
  irregular kernels.
\newblock {\em Anal. PDE}, 9(3):727--772, 2016.

\bibitem[ST2016]{ST2016}
R. Shvydkoy and E. Tadmor,
Eulerian dynamics with a commutator forcing,
arXiv:1612.04297.

\bibitem[TT2014]{TT2014}
Eitan Tadmor and Changhui Tan.
\newblock Critical thresholds in flocking hydrodynamics with non-local
  alignment.
\newblock {\em Philos. Trans. R. Soc. Lond. Ser. A Math. Phys. Eng. Sci.},
  372(2028):20130401, 22, 2014.
\end{thebibliography}
\end{document}